\documentclass[12pt]{amsart}
\usepackage[english]{babel}
\usepackage{txfonts}
\usepackage{fancyhdr}

\usepackage{amsthm,amsmath,amssymb,amsfonts,latexsym,url}

\usepackage{tikz}
\usepackage[all]{xy}
\usetikzlibrary{decorations.markings}
\usetikzlibrary{backgrounds,shapes}
\usepackage{subfig}
  \usepackage{graphicx}
\usepackage{enumerate}
\makeatletter
\let\@@enum@org\@@enum@
\def\@@enum@[#1]{\@@enum@org[\normalfont #1]}
\makeatother
\newtheorem{thm}{Theorem}[section]
\newtheorem{lem}[thm]{Lemma}
\newtheorem{cor}[thm]{Corollary}
\newtheorem{prop}[thm]{Proposition}

\newtheorem*{claim*}{Claim}

\theoremstyle{remark}

\newtheorem*{hint}{Hint}

\theoremstyle{definition}
\newtheorem{defn}[thm]{Definition}
\newtheorem{prob}{Problem}[section]

\newcommand\be{\begin{enumerate}}
\newcommand\ee{\end{enumerate}}
\newcommand\bp{\begin{proof}}
\newcommand\ep{\end{proof}}
\newcommand\bpp{\begin{prop}}
\newcommand\epp{\end{prop}}
\newcommand\bpb{\begin{prob}}
\newcommand\epb{\end{prob}}
\newcommand\bd{\begin{defn}}
\newcommand\ed{\end{defn}}
\newcommand\bh{\begin{hint}}
\newcommand\eh{\end{hint}}

\newcommand\N{\mathbb{N}}

\newcommand\Q{\mathbb{Q}}
\newcommand\C{\mathbb{C}}
\newcommand\Z{\mathbb{Z}}

\newcommand\GL{\operatorname{GL}}
\newcommand\PGL{\operatorname{PGL}}

\newcommand\SO{\operatorname{SO}}

\newcommand\gam{\Gamma}
\newcommand\bZ{\mathbb{Z}}

\newcommand\Mod{\operatorname{Mod}}

\newcommand\mC{\mathcal{C}}

\newcommand\rk{rk }

\def\thetitle{{Quotients of surface groups and homology of finite covers via quantum representations}}
\def\theshorttitle{}
\usepackage{hyperref}

\begin{document}
\raggedbottom
\title[\theshorttitle]\thetitle
\date{\today}

\author{Thomas Koberda}
\address{Department of Mathematics, University of Virginia, Charlottesville, VA 22904-4137, USA}
\email{thomas.koberda@gmail.com}

\author{Ramanujan Santharoubane}
\address{Institut de Math\'ematiques de Jussieu (UMR 7586 du CNRS), Equipe Topologie et G\'eom\'etrie Alg\'ebriques, Case 247, 4 pl. Jussieu, 75252 Paris Cedex 5, France}
\email{ramanujan.santharoubane@imj-prg.fr}

\begin{abstract}
We prove that for each sufficiently complicated orientable surface $S$, there exists an infinite image linear representation $\rho$ of $\pi_1(S)$ such that if $\gamma\in\pi_1(S)$ is freely homotopic to a simple closed curve on $S$, then $\rho(\gamma)$ has finite order. Furthermore, we prove that given a sufficiently complicated orientable surface $S$, there exists a regular finite cover $S'\to S$ such that $H_1(S',\Z)$ is not generated by lifts of simple closed curves on $S$, and we give a lower bound estimate on the index of the subgroup generated by lifts of simple closed curves. We thus answer two questions posed by Looijenga, and independently by Kent, Kisin, March\'e, and McMullen. The construction of these representations and covers relies on quantum $\text{SO}(3)$ representations of mapping class groups.
\end{abstract}

\maketitle

\section{Introduction}
\subsection{Main results}
Let $S=S_{g,n}$ be an orientable surface of genus $g\geq 0$ and $n\geq 0$ boundary components, which we denote by $S_{g,n}$. A \emph{simple closed curve} on $S$ is an essential embedding of the circle $S^1$ into $S$. We will call an element $1\neq g\in\pi_1(S)$ \emph{simple} if there is a simple representative in its conjugacy class. Note that contrary to a common convention, we are declaring curves which are freely homotopic to boundary components to be simple.

Our main result is the following:

\begin{thm}\label{thm:infinite image}
Let $S=S_{g,n}$ be a surface of genus $ g$ and with $n $ boundary components, excluding the $(g,n)$ pairs $\{(0,0), (0,1), (0,2), (1,0)\}$. There exists a linear representation \[\rho\colon\pi_1(S)\to\GL_d(\C)\] such that:
\begin{enumerate}
\item
The image of $\rho$ is infinite.
\item
If $g\in\pi_1(S)$ is simple then $\rho(g)$ has finite order.
\end{enumerate}
\end{thm}

Thus, Theorem \ref{thm:infinite image} holds for all surface groups except for those which are abelian, in which case the result obviously does not hold. In Subsection \ref{subsect:dimension}, we will give an estimate on the dimension of the representation $\rho$.

We will show that the representation $\rho$ we produce in fact contains a nonabelian free group in its image (see Corollary \ref{cor:free group2}, cf. \cite{FunarKohno}).

If $S'\to S$ is a finite covering space, we define the subspace \[H_1^{s}(S',\bZ)\subseteq H_1(S',\bZ)\] to be \emph{simple loop homology} of $S'$. Precisely, let $g\in\pi_1(S)$ be simple, and let $n(g)$ be the smallest positive integer such that $g^{n(g)}$ lifts to $S'$. Then \[H_1^{s}(S',\bZ):=\langle [g^{n(g)}]\mid g\in\pi_1(S) \textrm{ simple}\rangle\subseteq H_1(S',\bZ),\] where $[g^{n(g)}]$ denotes the homology class of $g^{n(g)}$ in $S'$. It is easy to check that $H_1^s(S,\bZ)=H_1(S,\bZ)$.

Identifying $\pi_1(S')$ with a subgroup of $\pi_1(S)$, we write $\pi_1^s(S')$ for the \emph{simple loop subgroup} of $\pi_1(S')$, which is generated by elements of the form $g^{n(g)}$. Here again, $g$ ranges over simple elements of $\pi_1(S)$. Observe that $H_1^s(S',\Z)$ is exactly the image of $\pi_1^s(S')$ inside of $H_1(S',\Z)$.

We obtain the following result as a corollary to Theorem \ref{thm:infinite image}:

\begin{thm}\label{thm:not gen homology}
Let $S= S_{g,n}$ be a genus $g$ surface with $n$ boundary components, excluding the $(g,n)$ pairs $\{(0,0), (0,1), (0,2), (1,0)\}$. Then there exists a finite cover $S'\to S$ such that \[H_1^s(S',\bZ)\subsetneq H_1(S',\bZ),\] i.e. the simple loop homology of $S'$ is properly contained in the full homology of $S'$.
\end{thm}

Again, Theorem \ref{thm:not gen homology} holds for all nonabelian surface groups, and in the abelian case the result obviously cannot hold. We remark that I. Irmer has proposed a version of Theorem \ref{thm:not gen homology} in \cite{Irmer}, and she proves that $H_1^s(S',\bZ)=H_1(S',\bZ)$ whenever the deck group $S'\to S$ is abelian.

We obtain the following immediate corollary from Theorem \ref{thm:not gen homology}:

\begin{cor}\label{cor:not gen pi1}
Let $S= S_{g,n}$ be a genus $g$ surface with $n$ boundary components, excluding the $(g,n)$ pairs $\{(0,0), (0,1), (0,2), (1,0)\}$. Then there exists a finite cover $S' \to S$ such that $\pi_1^s(S')$ is properly contained in $\pi_1(S')$.
\end{cor}

The representation $\rho$ in Theorem \ref{thm:infinite image} is produced using the quantum $\SO(3)$ representations of the mapping class group of $S$. We then use the Birman Exact Sequence to produce representations of the fundamental group of $S$. We use integral TQFT representations in order to approximate the representation $\rho$ by finite image representations $\{\rho_k\}_{k\in\N}$ which converge to $\rho$ in some suitable sense. Since the representations $\{\rho_k\}_{k\in\N}$ each have finite image, each such homomorphism classifies a finite cover $S_k\to S$. The cover $S'\to S$ furnished in Theorem \ref{thm:not gen homology} is any one of the covers $S_k\to S$, where $k\gg 0$.

Thus, the covers coming from integral TQFT representations will produce an infinite sequence of covers for which $H_1^s(S_k,\bZ)$ is a proper subgroup of $H_1(S_k,\bZ)$. It is unclear whether $H_1^s(S_k,\bZ)$ has finite or infinite index inside of $H_1(S_k,\bZ)$, or equivalently if Theorem \ref{thm:not gen homology} holds when integral coefficients are replaced by rational coefficients. However we will show that for a fixed $S$, the index of $H_1^s(S_k,\Z)$ in $H_1(S_k,\Z)$ can be arbitrarily large. I.~Agol has observed that Theorem \ref{thm:not gen homology} holds for the surface $S_{0,3}$ even when integral coefficients are replaced by rational coefficients (see \cite{KentMarche MO}). We will show that if Theorem \ref{thm:not gen homology} holds over $\Q$, then it in fact implies Theorem \ref{thm:infinite image}: see Proposition \ref{prop:homology quotient}.

\subsection{Notes and references}
The question of whether the homology of a regular finite cover $S'\to S$ is generated by pullbacks of simple closed curves on $S$ appears to be well--known though not well--documented in literature (see for instance \cite{KentMarche MO} and \cite{Looijenga}, and especially the footnote at the end of the latter). The problem itself is closely related to the congruence subgroup conjecture for mapping class groups (\cite{KentMarche MO}, \cite{Boggi1}, \cite{Boggi2}), to the virtually positive first Betti number problem for mapping class groups (see \cite{PutmanWieland}, and also \cite{FarbHensel} for a free group--oriented discussion), and to the study of arithmetic quotients of mapping class groups (see  \cite{grunewaldlubotzkymalestein}). It appears to have been resistant to various ``classical" approaches up to now.

The problem of finding an infinite image linear representation of a surface group in which simple closed curves have finite order has applications to certain arithmetic problems, and the question was posed to the authors by M. Kisin and C. McMullen (see also Questions 5 and 6 of \cite{Looijenga}). The existence of such a representation for the free group on two generators is an unpublished result of O. Gabber. Our work recovers Gabber's result, though our representation is somewhat different. In general, the locus of representations \[X_s\subset \mathcal{R}(\pi_1(S)),\GL_d(\C))\] of the representation variety of $\pi_1(S)$ which have infinite image but under which every simple closed curve on $\pi_1(S)$ has finite image is invariant under the action of $\text{Aut}(\pi_1(S))$. Theorem \ref{thm:infinite image} shows that this locus is nonempty, and it may have interesting dynamical properties. For instance, Kisin has asked whether any infinite image representation in $\mathcal{R}(\pi_1(S))$ has a finite orbit under the action of $\text{Aut}(\pi_1(S))$, up to conjugation by elements of $\pi_1(S)$. The locus $X_s$ is a good candidate in the search for such representations.

At least on a superficial level, Theorem \ref{thm:infinite image} is related to the generalized Burnside problem, i.e. whether or not there exist infinite, finitely generated, torsion groups. The classical Selberg's Lemma (see \cite{Raghunathan} for instance) implies that any finitely generated linear group has a finite index subgroup which is torsion--free, so that a finitely generated, torsion, linear group is finite. In the context of Theorem \ref{thm:infinite image}, we produce a finitely generated linear group which is not only generated by torsion elements (i.e. the images of finitely many simple closed curves), but the image of every element in the mapping class group orbit of these generators is torsion.

Finally, we note that G. Masbaum used explicit computations of TQFT representations to show that, under certain conditions, the image of quantum representations have an infinite order element (see \cite{Masbaum infinite}). This idea was generalized in \cite{AMU} where Andersen, Masbaum, and Ueno conjectured that a mapping class with a pseudo-Anosov component will have infinite image under a sufficiently deep level of the TQFT representations. In \cite{AMU}, the authors prove their conjecture for a four--times punctured sphere (cf. Theorem \ref{thm:sphere}). It turns out that their computation does not imply Theorem \ref{thm:sphere}, for rather technical reasons.

Although the papers \cite{Masbaum infinite} and \cite{AMU} do not imply our results, these computations of explicit mapping classes whose images under the TQFT representations are of infinite order is similar in spirit to our work in this paper.
 
 The reader may also consult the work of L. Funar (see \cite{Funar}) who proved independently (using methods different from those of Masbaum) that, under certain conditions, the images of quantum representations are infinite. Again, Funar's work does not imply our results.

\section{Acknowledgements}
The authors thank B. Farb, L. Funar, M. Kisin, V. Krushkal, E. Looijenga, J. March\'e, G. Masbaum, and C. McMullen for helpful conversations. The authors thank the hospitality of the Matematisches Forschungsinstitut Oberwolfach during the workshop ``New Perspectives on the Interplay between Discrete Groups in Low-Dimensional Topology and Arithmetic Lattices", where this research was initiated. The authors are grateful to an anonymous referee for a report which contained many helpful comments.

\section{Background}\label{sect:background}
In this section we give a very brief summary of facts we will require from the theory of TQFT representations of mapping class groups. We have included references for the reader to consult, but in the interest of brevity, we have kept the discussion here to a minimum.

\subsection{From representations of mapping class groups to representations of surface groups} \label{point pushing}
Let $S$ be an oriented surface with or without boundary and let $x_0\in S$ be a marked point, which we will assume lies in the interior of $S$. Recall that we can consider two mapping class groups, namely $\Mod(S)$ and $\Mod^1(S)=\Mod^1(S,x_0)$, the usual mapping class group of $S$ and the mapping class group of $S$ preserving the marked point $x_0$, respectively. By convention, we will require that mapping classes preserve $\partial S$ pointwise.

When the Euler characteristic of $S$ is strictly negative, these two mapping class groups are related by the Birman Exact Sequence (see \cite{Birman} or \cite{farbmargalit}, for instance): \[1\to\pi_1(S,x_0)\to\Mod^1(S)\to\Mod(S)\to 1.\] Thus, from any representation of $\Mod^1(S)$, we obtain a representation of $\pi_1(S)$ by restriction.  The subgroup $\pi_1(S)\cong\pi_1(S,x_0)<\Mod^1(S)$ is called the \emph{point--pushing subgroup} of $\Mod^1(S)$.

When $S$ has a boundary component $B$, one can consider the \emph{boundary--pushing subgroup} of $\Mod(S)$. There is a natural operation on $S$ which caps off the boundary component $B$ and replaces it with a marked point $b$, resulting in a surface $\hat{S}$ with one fewer boundary components and one marked point. There is thus a natural map $\Mod(S)\to\Mod^1(\hat{S},b)$, whose kernel is cyclic and generated by a Dehn twist about $B$. The boundary--pushing subgroup $\text{BP}(S)$ of $\Mod(S)$ is defined to be the preimage of the point--pushing subgroup of $\Mod^1(\hat{S},b)$. In general whenever $\hat{S}$ has negative Euler characteristic, we have an exact sequence \[1\to\Z\to\text{BP}(S)\to\pi_1(\hat{S},b)\to 1.\] The left copy of $\Z$ is central, and this extension is never split if $\hat{S}$ is closed and has negative Euler characteristic. In fact, $\text{BP}(S)$ is isomorphic to the fundamental group of the unit tangent bundle of $\hat{S}$. The reader is again referred to \cite{Birman} or \cite{farbmargalit} for more detail.

\begin{lem}\label{lem:simple finite}
Let $\rho\colon\Mod^1(S)\to Q$ be a quotient such that for each Dehn twist $T\in\Mod^1(S)$, we have $\rho(T)$ has finite order in $Q$. Then for every simple element $g$ in the point--pushing subgroup of $\Mod^1(S)$, we have that $\rho(g)$ has finite order in $Q$.
\end{lem}
\begin{proof}
Let $\gamma$ be an oriented simple loop in $S$ based at $x_0$. Identifying $\gamma$ with a simple or boundary parallel element $g$ of the point--pushing subgroup of $\Mod^1(S)$, we can express $g$ as a product of two Dehn twists thus: let $\gamma_1,\gamma_2\subset S$ be parallel copies of the loop $\gamma$, separated by the marked point $x_0$ (i.e. one component of $S\setminus\{\gamma_1,\gamma_2\}$ is an annulus containing the marked point $x_0$). Then the point pushing map about $\gamma$ is given, up to a sign, by $g=T_{\gamma_1}T^{-1}_{\gamma_2}$. Since $\gamma_1$ and $\gamma_2$ are disjoint, the corresponding Dehn twists commute with each other. Since $\rho(T_{\gamma_i})$ has finite order for each $i$, the element $\rho(g)$ has finite order as well.
\end{proof}

We will see in the sequel that if $\rho$ is a TQFT representation of $\Mod^1(S)$, then by Theorem \ref{thm:dt finite} below, Lemma \ref{lem:simple finite} applies.

\subsection{$\text{SO}(3)$--TQFT representations}\label{subsect:su2 tqft}
Let $S_{g,n}$ be a genus $g$ closed, oriented surface with $n$ boundary components. The $\SO(3)$ topological quantum field theories (TQFTs) take as an input an odd integer $p \geq 3$ and a $2p^{th}$ primitive root of unity. As an output, they give a projective representation \[\rho_p \colon \Mod^1(S_{g,n})\to\PGL_{d}(\C),\] which moreover depends on certain coloring data (which will be specified later on), and where here the dimension $d$ depends the input data. The notion of a TQFT was introduced by Witten (see \cite{Witten}). His ideas were based on a physical interpretation of the Jones polynomial involving the Feynman path integral, and the geometric quantization of the $3$--dimensional Chern--Simons theory. The first rigorous construction of a TQFT was carried out by Reshetikhin and Turaev, using the category of semisimple representations of the universal enveloping algebra for the quantum Lie algebra $\mathrm{SL}(2)_q$ (see \cite{ReshetikhinTuraev} and \cite{Turaev}). We will work in the TQFT constructed by Blanchet, Habegger, Masbaum, and Vogel in~\cite{bhmv}, wherein an explicit representation associated to a TQFT is constructed using skein theory. Perhaps the most important feature of these representations is the following well--known fact (see~\cite{bhmv}):

\begin{thm}\label{thm:dt finite}
Let $T\in\Mod^1(S_{g,n})$ be a Dehn twist about a simple closed curve. Then $\rho_p(T)$ is a finite order element of $\PGL_{d}(\C)$.
\end{thm}
It is verifying that certain mapping classes have infinite order under TQFT representations which is often nontrivial and makes up most of the content of this paper. Here and in Subsection \ref{subsect:integral tqft} we will survey some basic properties and computational methods for TQFT representations which we will require.

One can define a certain cobordism category $\mC$ of closed surfaces with colored banded points, in which the cobordisms are decorated by uni-trivalent colored banded graphs. The details of this category are not essential to our discussion; for details we direct the reader to~\cite{bhmv}. The $\SO(3)$--TQFT is a functor $Z_p$ from the category $\mC$ to the category of finite dimensional vector spaces over $\C$.

A \emph{banded point} (or an \emph{ribbon point}) on a closed oriented surface is an oriented submanifold  which is homeomorphic to the unit interval. If a surface has multiple banded points, we will assume that these intervals are disjoint.
 A banded point provides a good substitute for a boundary component within a closed surface, and a simple loop on $S$ which encloses a single banded point can be thought of a boundary parallel loop. When one wants to study a surface with boundary from the point of view of TQFTs, one customarily attaches a disk to each boundary component and places a single banded point in the interior of each such disk. The banded points are moreover colored, which is to say equipped with an integer.
 
 By capping off boundary components, we can start with a surface $S_{g,n}$ and produce a closed surface $\hat{S}_{g,n}$ equipped with $n$ colored banded points. We will include a further colored banded point $x$ in the interior of the surface with boundary $S_{g,n}$, which will play the role of a basepoint. We denote by $(\hat{S}_{g,n},x)$ the closed surface with $n+1$ colored banded points thus obtained.

 Now for $p \geq 3$ odd, the $\SO(3)$--TQFT defines a finite dimensional vector space $$V_{p}(\hat{S}_{g,n},x).$$
For the sake of computations, it is useful to write down an explicit basis for the space $V_{p}(\hat{S}_{g,n},x)$. 

Denote by $y$ the set of $n+1$ colored banded points on $(\hat{S}_{g,n},x)$, and by $S_g$ the underlying closed surface without colored banded points. Let $\mathcal{H}$ be a handlebody such that $\partial{\mathcal{H}} = S_g$, and let $G$ be a uni-trivalent banded graph such that $\mathcal{H}$ retracts to $G$. We suppose that $G$ meets the boundary of $\mathcal{H}$ exactly at the banded points $y$ and this intersection consists exactly of the degree one ends of $G$. A $p$--admissible coloring of $G$ is a coloring, i.e. an assignment of an integer, to each edge of $G$ such that at each degree three vertex $v$ of $G$, the three (non--negative integer) colors $\{a,b,c\}$ coloring edges meeting at $v$ satisfy the following conditions:
\begin{enumerate}
\item
$|a-c|\leq b\leq a+c$;
\item
$a+b+c\leq 2p-4$;
\item
each color lies between $0$ and $p-2$;
\item
the color of an edge terminating at a banded point $y_i$ must have the same color as $y_i$.
\end{enumerate}
To any $p$-admissible coloring $c$ of $G$, there is a canonical way to associate an element of the skein module \[S_{A_p}(\mathcal{H},(\hat{S}_{g,n},x)),\] where here the notation refers to the usual skein module with the indeterminate evaluated at $A_p$, which in turn is a $2p^{th}$ primitive root of unity. The skein module element is produced by cabling the edges of $G$ by appropriate Jones-Wenzl idempotents (see \cite[Section 4]{bhmv} for more detail). If moreover all the colors are required to be even, it turns out that the vectors associated to $p$--admissible colorings give a basis for $V_p(\hat{S}_{g,n},x)$ (see \cite[Theorem 4.14]{bhmv}).

We can now coarsely sketch the construction of TQFT representations. If $x_0 \in x$, we may contract $x$ down to a point in order to obtain a surface with a marked point $(S_{g,n},x_0)$. To construct this representation, one takes a mapping class $f\in\Mod^1(S_{g,n},x_0)$ and one considers the mapping cylinder of $f^{-1}$. The TQFT functor gives as an output a linear automorphism $\rho_p(f)$ of $V_{p}(\hat{S}_{g,n},x)$.

This procedure gives us a projective representation $$ \rho_p : \Mod^1(S_{g,n},x_0) \to \text{PAut}(V_p(S_{g,n},x)),$$ since the composition law is well--defined only up to multiplication by a root of unity.

Our primary interest in the representation $\rho_p$ lies in its restriction to the point pushing subgroup of $\Mod^1(S_{g,n},x_0)$. More precisely, we need to compute explicitly the action of a given element of $\pi_1(S_{g,n},x_0)$ on the basis of $V_{p}(\hat{S}_{g,n},x)$ described above. We will explain how to do these calculations in Subsection \ref{Q rep surface group}, and some illustrative examples will be given in the proof of Theorem \ref{thm:sphere}.

\subsection{Integral TQFT representations}\label{subsect:integral tqft}
In this Subsection we follow some of the introductory material of~\cite{GilmerMasbaum} and of~\cite{GilmerMasbaumTorus}. The TQFT representations of mapping class groups discussed in Subsection \ref{subsect:su2 tqft} are defined over $\C$ and may not have good integrality properties. In our proof of Theorem \ref{thm:not gen homology}, we will require an ``integral refinement" of the TQFT representations which was constructed by Gilmer~\cite{Gilmer} and Gilmer--Masbaum~\cite{GilmerMasbaum}. These integral TQFT representations have all the properties of general $\SO(3)$--TQFTs which we will require, and in particular Theorem \ref{thm:infinite image} holds for them.

To an oriented closed surface $S$ equipped with finitely many colored banded points and an odd prime $p$ (with $p \equiv 3 \, \pmod 4$) , one can associate a free, finitely generated module over the ring of integers $\mathcal{O}_p=\Z[\zeta_p]$, where $\zeta_p$ is a primitive $p^{th}$ root of unity. We will denote this module by $\mathcal{S}_p(S)$. This module is stable by the action of the mapping class group, and moreover tensoring this action with $\C$ gives us the usual $\SO(3)$--TQFT representation.

Most of the intricacies of the construction and the properties of the integral TQFT representations are irrelevant for our purposes, and we therefore direct the reader to the references mentioned above for more detail. We will briefly remark that the construction of the integral TQFT representations is performed using the skein module, so that the computations in integral TQFT are identical to those in the TQFTs described in Subsection \ref{subsect:su2 tqft}. The feature of these representations which we will require is the following filtration by finite image representations.

Let $h=1-\zeta_p$, which is a prime in $\Z[\zeta_p]$. We consider the modules \[\mathcal{S}_{p,k}(S)=\mathcal{S}_p(S)/h^{k+1}\mathcal{S}_p(S),\] which are finite abelian groups for each $k \geq 0$. The representation $\rho_p$ has a natural action on $\mathcal{S}_{p,k}(S)$, and we denote the corresponding representation of the mapping class group by $\rho_{p,k}$. 

Observe (see \cite{GilmerMasbaumTorus}) that the natural map \[\mathcal{O}_p\to\varprojlim \mathcal{O}_p/h^{k+1}\mathcal{O}_p\] is injective, where here the right hand side is sometimes called the \emph{$h$--adic completion} of $\mathcal{O}_p$. We immediately see that \[\bigcap_k \ker{\rho_{p,k}} = \ker{\rho_p}.\]

\section{Infinite image TQFT representations of surface groups}

\subsection{Quantum representations of surface groups}\label{Q rep surface group}
In this Subsection we describe the key idea of this paper, which is the use of the Birman Exact Sequence (see Subsection \ref{point pushing}) together with TQFT representations in order to produce exotic representations of surface groups. Using the notation of Subsection \ref{subsect:su2 tqft}, we have a projective representation \[ \rho_p : \Mod^1({S}_{g,n},x_0) \to \text{PAut}(V_{p}(\hat{S}_{g,n},x)).\]

Restriction to the point--pushing subgroup gives a projective representation of the fundamental group of $S_{g,n}$:  $$ \rho_p : \pi_1({S}_{g,n},x_0) \to \text{PAut}(V_{p}(\hat{S}_{g,n},x)),$$ which is defined whenever the Euler characteristic of $S_{g,n}$ is strictly negative. In order to make the computations more tractable, let us describe this action more precisely. Let \[ \gamma : [0,1] \to S_{g,n}\]  be a loop based at $x_0$. By the Birman Exact Sequence, $\gamma$ can be seen as a diffeomorphism $f_{\gamma}$ of $(S_{g,n},x_0)$. Pick a lift $\tilde{f}_{\gamma}$ of $f_{\gamma}$ to the \emph{ribbon mapping class group} of $(S_{g,n},x)$ , i.e. the mapping class group preserving the orientation on the banded point $x$, together with any coloring data that might be present. Note that any two lifts of $f_{\gamma}$ differ by a twist about the banded point $x$, which is a central element of the ribbon mapping class group. The preimage of the point pushing subgroup inside of the ribbon mapping class group is easily seen to be isomorphic to the boundary pushing subgroup $\text{BP}(S_{g,{n+1}})$, where the extra boundary component is the boundary of a small neighborhood of the banded point $x$.

By definition, $\tilde{f}_{\gamma}^{-1}$ is isotopic to the identity via an isotopy which is allowed to move $x$. Following the trajectory of the colored banded point $x$ in this isotopy gives a colored banded tangle $\tilde{\gamma}$ inside $ S_g \times [0,1] $. Observe that $\tilde{\gamma}$ is just a thickening of the tangle $t \in [0,1] \mapsto (\gamma(1-t),t) \in S_{g,n} \times [0,1]$.

We can naturally form the decorated cobordism $C_{\tilde{\gamma}}$ by considering $S_g \times [0,1]$ equipped with the colored banded tangle $$(\tilde{\gamma},a_1 \times [0,1],\ldots, a_{n} \times [0,1]),$$ where $(a_1,\ldots,a_n)$ are the $n$ colored banded points on $\hat{S}_{g,n}$. By applying the TQFT functor, we obtain an automorphism \[ \rho_p(\gamma)  := \rho_p(\tilde{f}_{\gamma}) = Z_p(C_{\tilde{\gamma}}). \]

We crucially note that this automorphism generally depends on the choice of the lift $\tilde{f}_{\gamma}$, but that changing the lift will only result in $\rho_p(\gamma)$ being multiplied by a root of unity, since the Dehn twist about $x$ acts by multiplication by a root of unity. In particular, different choices of lift give rise to the same element $\text{PAut}(V_p(S_{g,n},x))$. Thus, we do indeed obtain a representation of $\pi_1(S_{g,n})$, and not of the boundary pushing subgroup.

We now briefly describe how to compute the action of the loop $\gamma$ on the basis of TQFT described in Subsection 
\ref{subsect:su2 tqft}. Let $\mathcal{H}$ be a handlebody with $\partial{\mathcal{H}} = S_g$ and let $G$ be a uni-trivalent banded graph such that $\mathcal{H}$ retracts to $G$. Let $G_c$ be a coloring of this graph as explained in Subsection \ref{subsect:su2 tqft}. In the TQFT language, applying $Z_p(C_{\tilde{\gamma}})$ to $G_c$ simply means that we glue the cobordism $C_{\tilde{\gamma}}$ to the handlebody $\mathcal{H}$ equipped with $G_c$. This gluing operation gives the same handlebody $\mathcal{H}$, but with a different colored banded graph inside. To express this new banded colored graph in terms of the basis mentioned above, we have to use a set of local relations on colored banded graphs, namely a colored version of the Kauffman bracket and the application of Jones--Wenzl idempotents. For the sake of brevity, we refer the reader to the formulas computed by Masbaum and Vogel (see \cite{MasbaumVogel}).

\subsection{The three-holed sphere}\label{subsect:sphere}
The main technical case needed to establish Theorem \ref{thm:infinite image} is the case of the surface $S_{0,3}$, a sphere with three boundary components. Let $\Mod^1(S_{0,3})$ be the mapping class group of $S_{0,3}$, preserving a fourth marked point in the interior of $S_{0,3}$.

The correct setup for TQFT representations of $\Mod^1(S_{0,3})$ is that of spheres with banded points. Let $(S^2,1,1,2,2)$ be a sphere equipped with two banded points colored by $1$ and two banded points colored by $2$ (where one of these two serves as the base point). Let $p \geq 5$ and $A_p$ be a $2p^{th}$ primitive root of unity. According to \cite[Theorem 4.14]{bhmv} and the survey given in Subsection \ref{subsect:su2 tqft}, the vector space $V_{p}(S^2,1,1,2,2)$ has a basis described by the following two colored banded graphs
\begin{equation} \label{eq1} G_a   = \begin{minipage}[c]{4cm}
\includegraphics[scale = 0.14]{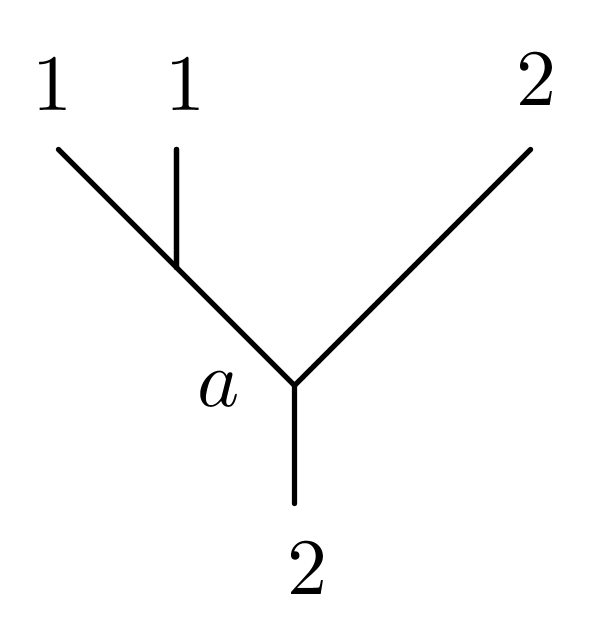}
\end{minipage} \text{with} \quad a \in \{ 0,2 \}, \end{equation}
where these graphs live in the $3$--ball whose boundary is $(S^2,1,1,2,2)$. The computations can be done in this basis as described in Subsection \ref{Q rep surface group}. In order to simplify the computations, we will work in a different basis. We note that these colored banded graphs can be expanded to banded tangles in the $3$-ball using a well--studied procedure (see \cite{MasbaumVogel} or \cite[section 4]{bhmv}). It is easy to check that the following elements form a basis of $V_p(S^2,1,1,2,2)$:
\[ u_1   = \begin{minipage}[c]{3cm}
\includegraphics[scale = 0.16]{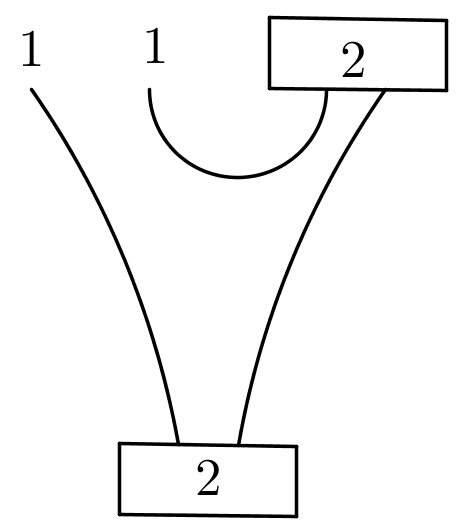}
\end{minipage}  \quad u_2  =  \begin{minipage}[c]{3.5cm}
\includegraphics[scale = 0.16]{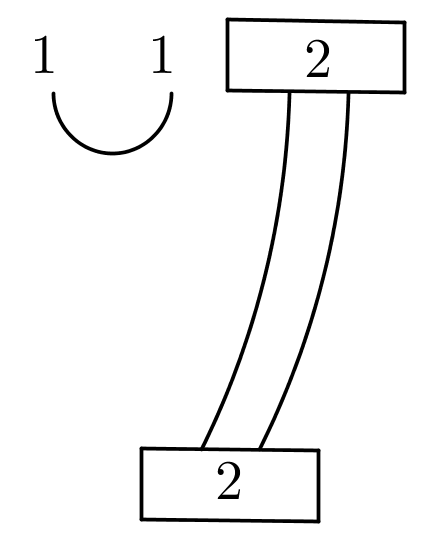}
\end{minipage} \]
More precisely, \[ u_1 = \dfrac{G_0}{A_p^2+A_p^{-2}}+G_2 \quad \text{and} \quad u_2 = G_0. \]
Here again, the ambient $3$--ball in which these tangles live is not drawn. The arcs drawn stand for banded arcs with the blackboard framing. The two end points of the arcs in the top left of the picture are attached to the points colored by $1$ on the boundary of the $3$--ball. The two rectangles labeled by $2$ represent the second Jones--Wenzl idempotent, and these are attached to the two points colored by $2$ on the boundary of the $3$--ball. The construction and most properties of the Jones--Wenzl idempotents are irrelevant for our purposes, and the interested reader is directed to \cite[Section 3]{bhmv}.

The TQFT representation of $\Mod^1(S_{0,3})$ furnishes a homomorphism \[\rho_p \colon \Mod^1(S_{0,3})\to\text{PAut}( V_{p}(S^2,1,1,2,2))\cong \PGL_2(\C)\] in whose image all Dehn twists have finite order by Theorem \ref{thm:dt finite}, which then by restriction gives us a homomorphism \[\rho_p\colon\pi_1(S_{0,3})\to\PGL_2(\C)\] in whose image all simple loops have finite order, by Lemma \ref{lem:simple finite}.

\begin{thm}\label{thm:sphere}
For all $p\gg 0$, the image of the representation \[\rho_p\colon\pi_1(S_{0,3})\to\PGL_2(\C)\] contains an element of infinite order.
\end{thm}
\begin{proof}
We can compute the action of $\pi_1(S_{0,3})$ on $ V_{p}(S^2,1,1,2,2)$ explicitly, and thus find an element of $g\in\pi_1(S_{0,3})$ such that $\rho_p(g)$ has infinite order.

Let $\{\gamma_1,\gamma_2,\gamma_3\}$ be the usual generators of the fundamental group of the three--holed sphere: 
\[\begin{minipage}[c]{4cm}
\includegraphics[scale = 0.13]{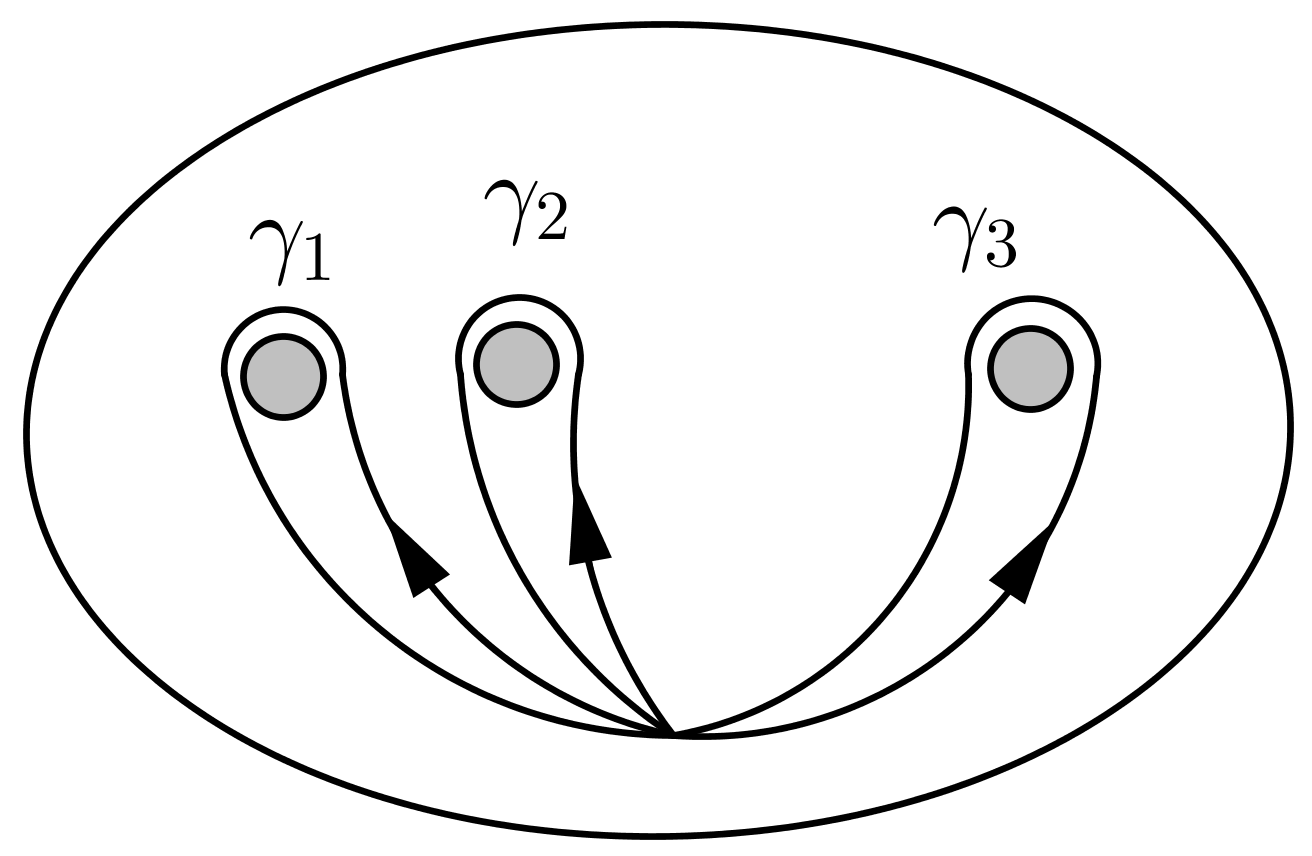}
\end{minipage}\]

Following the discussion in Subsection \ref{Q rep surface group}, we can write down matrices for the action of $\rho(\gamma_i)$ for each $i$. A graphical representation is as follows:
 \[\rho_p(\gamma_1) u_2 \quad = \begin{minipage}[c]{3.5cm}
\includegraphics[scale = 0.15]{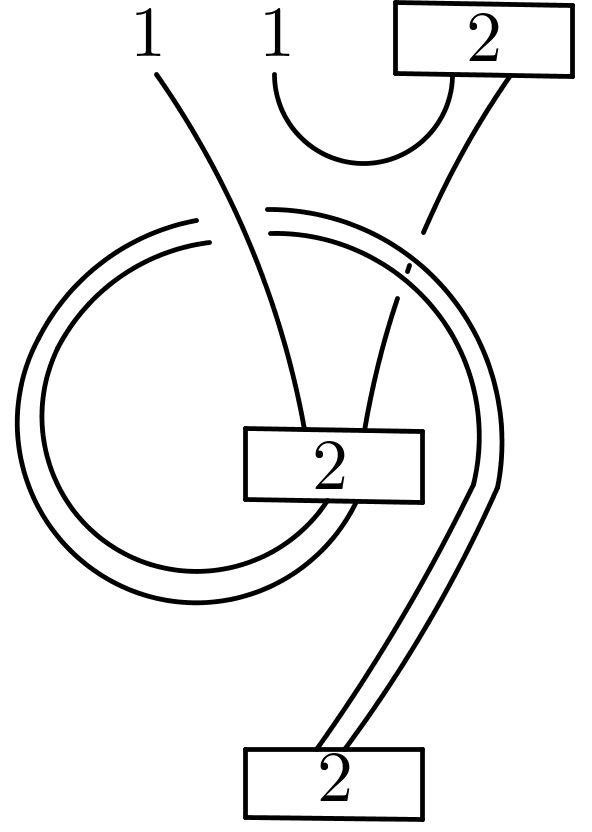}
\end{minipage} \quad  \rho_p(\gamma_1) u_2 \quad = \begin{minipage}[c]{3.5cm}
\includegraphics[scale = 0.15]{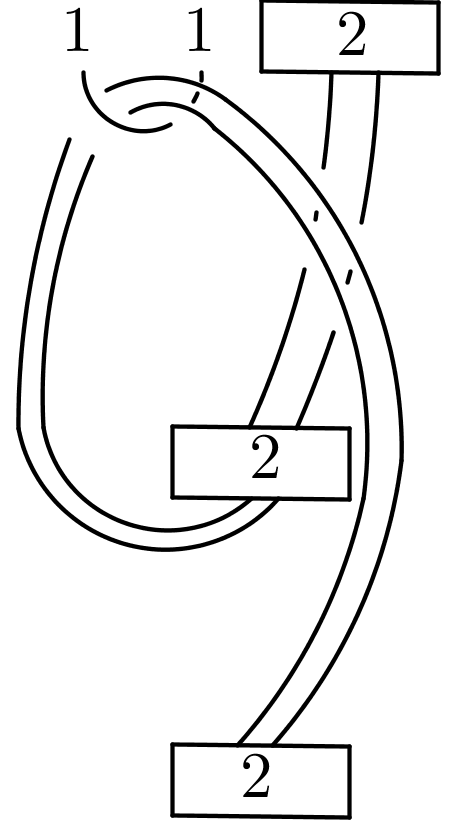}
\end{minipage}\] 
 \[\rho_p(\gamma_2) u_1 \quad = \begin{minipage}[c]{3.5cm}
\includegraphics[scale = 0.15]{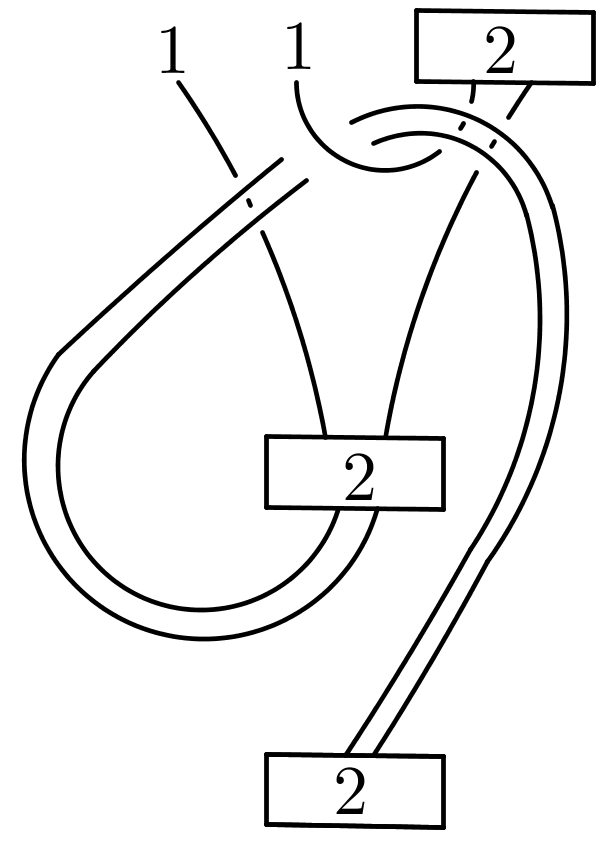}
\end{minipage} \quad  \rho_p(\gamma_2) u_2 \quad = \begin{minipage}[c]{3.5cm}
\includegraphics[scale = 0.15]{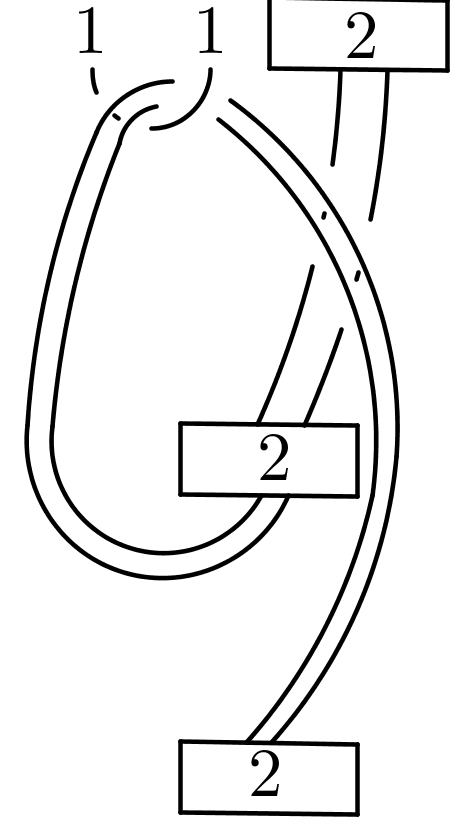}
\end{minipage}\]We can then reduce these diagrams using the skein relations \begin{align} \begin{minipage}[c]{1cm}
\includegraphics[scale = 0.09]{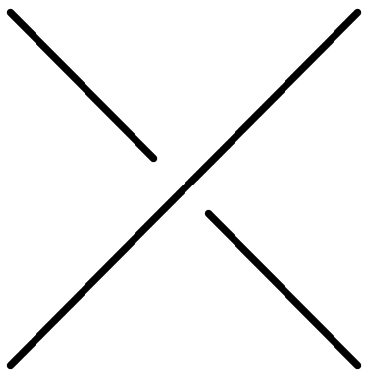} \end{minipage}&= A_p \quad  \begin{minipage}[c]{1cm}
\includegraphics[scale = 0.09]{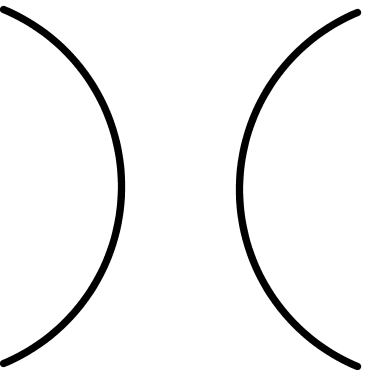} \end{minipage} + A^{-1}_p \quad \begin{minipage}[c]{1cm}
\includegraphics[scale = 0.09]{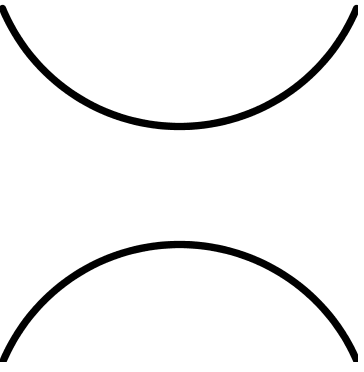} \end{minipage} \notag \\
\begin{minipage}[c]{1cm}
\includegraphics[scale = 0.06]{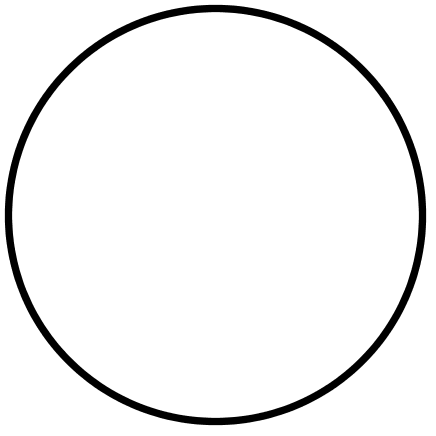} \end{minipage} & = -A^2_p-A^{-2}_p \notag 
\end{align}
 in order to obtain diagrams without crossing and without trivial circles. We then use the Jones--Wenzl idempotents, and in particular the rule 
\[ \begin{minipage}[c]{1cm}
\includegraphics[scale = 0.18]{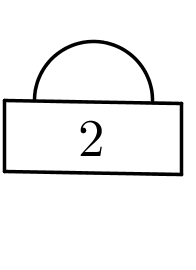}
\end{minipage}  = 0,\]
in order to simplify the diagrams further. We thus obtain for each diagram a linear combination of the tangles $\{u_1,u_2\}$. One easily checks that we have the following matrices:
\[ \rho_p(\gamma_1)  = \begin{pmatrix}
       1 & A_p^{-10}-A_p^{-2} \\
 0 & A_p^{-12}
\end{pmatrix} \quad  \rho_p(\gamma_2) = \begin{pmatrix}
    A_p^{-8} & A_p^{2}-A_p^{-6} \\
 A_p^{-10}-A_p^{-14} & 1-A_p^{-8}+A_p^{-12}
\end{pmatrix}. \] \\
Similarly, we can compute \[\rho_p(\gamma_3)  = \begin{pmatrix}
  A_p^{-8} & 0 \\
  -A_p^{-2}+A_p^{-6} & 1 \\
\end{pmatrix}.\]
\\

Here we recall that $p \geq 5$ is an odd integer and $A_p$ is $2p^{th}$ primitive root of unity.
Now, one checks that $ \text{tr}(\rho_p(\gamma_1) \rho_p(\gamma_2)^{-1}) = A_p^{12}-A_p^{4}+2-A_p^{-4}+A_p^{-12}$. So we have that \[| \text{tr}(\rho_p(\gamma_1 \gamma_2^{-1}))|  \underset{A_p \to e^{\frac{i \pi}{6}}}{\longrightarrow} 5 > 2.\] If we take a sequence of $2p^{th}$ primitive roots of unity $\{A_{p}\}$ such that $A_{p}\to e^{\frac{i \pi}{6}}$ as $p\to\infty$, we see that  $| \text{tr}(\rho_{p}(\gamma_1 \gamma_2^{-1}))| >2$ for $p \gg 0$. Whenever $| \text{tr}(\rho_{p}(\gamma_1 \gamma_2^{-1}))| >2$, an elementary calculation shows that $\rho_{p}(\gamma_1 \gamma_2^{-1})$ has an eigenvalue which lies off the unit circle. It follows that no power of $\rho_{p}(\gamma_1 \gamma_2^{-1})$ is a scalar matrix, since the determinant of this matrix is itself a root of unity. Thus, $\rho_{p}(\gamma_1 \gamma_2^{-1})$ has infinite order for $p \gg 0$.
\end{proof}

In the proof above, we note that the based loop $\gamma_1\gamma_2^{-1}$ is freely homotopic to a ``figure eight" which encircles two punctures, and such a loop is not freely homotopic to a simple loop. In light of the discussion above, Theorem \ref{thm:sphere} could be viewed as giving another proof that $\gamma_1\gamma_2^{-1}$ is in fact not represented by a simple loop.

\begin{cor}\label{cor:free group}
Let $\rho_p$ be as above. For $p\gg 0$, the image of $\rho_p$ contains a nonabelian free group.
\end{cor}
\begin{proof}
Let $g=\gamma_1\gamma_2^{-1}$ as in the proof of Theorem \ref{thm:sphere}, and let $p\gg0$ be chosen so that $\rho_p(g)$ has infinite order. We have shown that $\rho_p(g)$ in fact admits an eigenvalue which does not lie on the unit circle, so that $\lambda=\rho_p(g)\in\PGL_2(\C)$ can be viewed as a loxodromic isometry of hyperbolic $3$--space. By a standard Ping--Pong Lemma argument, it suffices to show that there exists a loxodromic isometry $\mu$ in the image of $\rho_p$ whose fixed point set on the Riemann sphere $\hat{\C}$ is disjoint from that of $\lambda$. Indeed, then sufficiently high powers of $\lambda$ and $\mu$ will generate a free subgroup of $\PGL_2(\C)$, which will in fact be a classical Schottky subgroup of $\PGL_2(\C)$. See \cite{delaHarpe}, for instance.

To produce $\mu$, one can just conjugate $\lambda$ by an element in the image of $\rho_p$. It is easy to check that for $p\gg0$, the element $\rho_p(\gamma_1)\in\PGL_2(\C)$ has two fixed points, namely $\infty$ and the point \[z_0=\frac{A_p^{-2}-A_p^{-10}}{1-A_p^{-12}},\] which is distinct from infinity if $A_p$ is not a twelfth root of unity. An easy computation using \[\lambda= \begin{pmatrix}
  A_p^{24}-A_p^4+2-A_p^{-4}-A_p^{-8}+A_p^{-12} & -A_p^{14}+A_p^6-A_p^2+A_p^{-6} \\
  -A_p^{-10}+A_p^{-14} & A_p^{-8} \\
\end{pmatrix}\] shows that $\lambda$ does not fix $\infty$. Similarly, a direct computation shows that $\lambda$ does not fix $z_0$. If $p\gg0$, the isometry $\rho_p(\gamma_1)$ will not preserve the fixed point set of $\lambda$ setwise, so that we can then conjugate $\lambda$ by a power of $\rho_p(\gamma_1)$ in order to get the desired $\mu$.
\end{proof}

\subsection{General surfaces}\label{subsect:surface}
In this Subsection, we bootstrap Theorem \ref{thm:sphere} in order to establish the main case of Theorem \ref{thm:infinite image}. 

\begin{proof}[Proof of Theorem \ref{thm:infinite image}]
Retaining the notation of Subsection \ref{Q rep surface group}, we start with the data of a surface $S_{g,n}$ which is equipped with a marked point $x_0$ in its interior. We then cap off the boundary components  to get $(\hat{S}_{g,n},x)$, which is a closed surface with $n+1$ colored banded points. Recall that $x$ is a thickening of the marked point $x_0$. Suppose furthermore that $x$ is colored by $2$.

 Restricting the quantum representation of $\Mod^1(S_{g,n},x_0)$ to the point pushing subgroup gives a representation  \[ \rho_p : \pi_1(S_{g,n},x_0) \to \text{PAut}(V_p(\hat{S}_{g,n},x)) \] Proving that $\rho_p$ has infinite image can be done by combining Theorem \ref{thm:sphere} with a standard TQFT argument which has already appeared in \cite{Masbaum infinite}. For the sake of concreteness, we now give this argument in the case $g \geq 2$, and the other cases covered by Theorem \ref{thm:infinite image} can be obtained in a similar fashion. We adopt the standing assumption that all banded points on $\hat{S}_{g,n}$ are colored by $2$.
 
 Consider the three-holed sphere inside $S_{g,n}$ whose boundary components are the two curves drawn in the following diagram: 
\[ \begin{minipage}[c]{3cm}
\includegraphics[scale = 0.14]{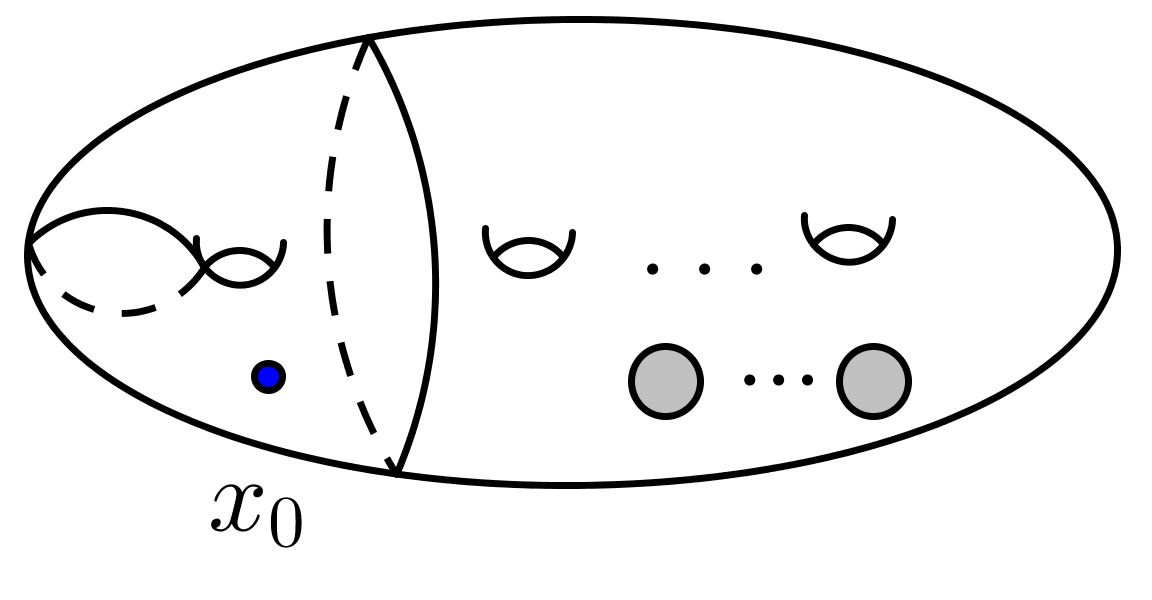}\end{minipage}\]
Thus, we can map $\pi_1(S_{0,3})$ into $\pi_1(S_{g,n})$ and get an action of $\pi_1(S_{0,3})$ on $V_p(\hat{S}_{g,n},x)$. From Theorem \cite[Theorem 1.14]{bhmv}, we see that this action of $\pi_1(S_{0,3})$ on $V_p(\hat{S}_{g,n},x)$ contains as a direct summand a vector space $V \otimes W$, where $\pi_1(S_{0,3})$ acts on $V$ as on $V_p(S^2,1,1,2,2)$ as discussed in Subsection \ref{subsect:sphere} and where $W$ is another representation. The conclusion of the theorem follows from observation that some element of $\pi_1(S_{0,3})$ acts with infinite order on \[V=V_p(S^2,1,1,2,2)\] for $p \gg 0$ by Theorem \ref{thm:sphere}, and thus this element also acts with infinite order on $V\otimes W$.
\end{proof}

We briefly remark that, as mentioned in Subsection \ref{subsect:integral tqft}, Theorem \ref{thm:infinite image} also holds for integral TQFT representations, with the same proof carrying over. Thus we have:
\begin{cor}\label{cor:integral infinite}
Let $p$ be an odd prime, let $\rho_p$ be the associated integral TQFT representation of $\Mod^1(S)$, and let $\pi_1(S)<\Mod^1(S)$ be the point--pushing subgroup. Then for all $p\gg 0$, we have that:
\begin{enumerate}
\item
The representation $\rho_p$ has infinite image.
\item
If $g\in\pi_1(S)$ is simple or boundary parallel then $\rho_p(g)$ has finite order.
\end{enumerate}
\end{cor}

A direct consequence of Corollary \ref{cor:free group} and the splitting argument in TQFT used in the proof of Theorem \ref{thm:infinite image} is the following fact:

\begin{cor}\label{cor:free group2}
The image of the representation \[ \rho_p : \pi_1(S_{g,n}) \to \text{PAut}(V_p(\hat{S}_{g,n},x)) \] contains a nonabelian free group for $p \gg 0$ as soon as $\pi_1(S_{g,n})$ is not abelian.

\end{cor} 

\subsection{Dimensions of the representations}\label{subsect:dimension}
In this subsection, we give a quick and coarse estimate on the dimension of the representation $\rho$ in Theorem \ref{thm:infinite image}. It suffices to estimate the dimension of the TQFT representation $\rho_p$ of $\Mod^1(S)$ which restricts to an infinite image representation of $\pi_1(S)$, and then embed the corresponding projective general linear group in a general linear group.

The dimension of the space $V_{p}(S,x)$ is given by a \emph{Verlinde formula}, which simply counts the number of even $p$-admissible colorings of trivalent ribbon graphs in a handlebody bounded by $S$, as sketched in Subsection \ref{subsect:su2 tqft}. In principle, it is possible to compute the dimension of $V_{p}(S,x)$, though a closed formula is often quite complicated. See~\cite{bhmv} and ~\cite{GilmerMasbaum} for instance. 

We note that, fixing $p$, the dimension of $V_{p}(S,x)$ grows exponentially in the genus of $S$, since in the proof of Theorem \ref{thm:infinite image} we have that the color assigned to each boundary component of $S$ is $2$.

The argument for Theorem \ref{thm:infinite image} furnishes one $p$ which works for all genera, since the infinitude of the image of $\rho_p$ is proven by considering the restriction of $\rho_p$ to a certain three--holed sphere inside of $S$. It follows that the target dimension of \[\rho_p\colon \pi_1(S)\to\PGL_d(\C)\] is $d\sim C^g$ for some constant $C>1$. Since $\PGL_d(\C)$ can be embedded in $\GL_{d^2}(\C)$ using the adjoint action, we obtain the following consequence:

\begin{cor}\label{cor:dimension}
There is a constant $C>1$ such that if $S=S_{g}$ is a closed surface of genus $g$, there is a representation $\rho$ of $\pi_1(S)$ satisfying the conclusions of Theorem \ref{thm:infinite image}, of dimension bounded by $C^g$.
\end{cor}

\section{Homology and finite covers}
In this section we use integral TQFT representations to prove Theorem \ref{thm:not gen homology} and Corollary \ref{cor:not gen pi1}.

\subsection{From infinite image representations to finite covers} 
 Before using integral TQFT to establish Theorem \ref{thm:not gen homology} we show how from any projective representation $\rho : \pi_1(S) \to \PGL_d(\C)$ satisfying the conclusions of Theorem \ref{thm:infinite image} we can build a covering $S' \to S$ satisfying the conclusions of Corollary \ref{cor:not gen pi1}. Thus, any representation $\rho : \pi_1(S) \to \PGL_d(\C)$ satisfying the conclusions of Theorem \ref{thm:infinite image}, even one not coming from TQFTs, already gives a somewhat counterintuitive result. More precisely:
 \begin{thm}
 Let $\rho : \pi_1(S) \to \PGL_d(\C)$ be a projective representation such that 
 \begin{enumerate}
\item The image of $\rho$ is infinite.
\item If $g \in \pi_1(S)$ is simple then $\rho(g)$ has finite order.
\end{enumerate} Then there exists a finite cover $S' \to S$ such that $\pi_1^s(S')$ is an infinite index subgroup of $\pi_1(S')$. 
\end{thm}
\begin{proof}
Let $\rho : \pi_1(S) \to \PGL_d(\C)$ as in the statement of the present theorem. The image of $\rho$ is an infinite, finitely generated, linear group. Using Selberg's Lemma, we can find a finite index torsion--free subgroup $H \vartriangleleft \rho(\pi_1(S))$. So $H' = \rho^{-1}(H) \vartriangleleft \pi_1(S)$ is a finite index subgroup which classifies a finite cover $S' \to S$.

Now let $g \in \pi_1(S)$ be  simple and let $n(g)$ be the smallest integer such that $g^{n(g)} \in H'$. We note that the element \[\rho(g^{n(g)}) \in \rho(H') = H\] is torsion since $\rho(g)$ has finite order, which forces $\rho(g^{n(g)}) = 1$, since $H$ is torsion--free. On the one hand, we have that $n(g)$ is precisely the order of $\rho(g)$, and that $\pi_1^s(S')< \ker{\rho}$. On the other hand $H' = \rho^{-1}(H)$ so $\ker{\rho}< H'$. But $\rho$ has infinite image, so that $\ker{\rho}$ is an infinite index subgroup of $\pi_1(S)$, from which we can conclude that $\pi_1^s(S')$ is an infinite index subgroup of $H'=\pi_1(S')$.
\end{proof}

\subsection{Homology of finite covers} \label{General surface}
Let $p \geq 7$ be an odd integer and let $S=S_{g,n}$ be compact surface as in the statement of Theorem \ref{thm:not gen homology}. Using the notation of Subsection \ref{Q rep surface group} we consider the representation $$\rho_p : \pi_1(S) \to \text{PAut}(V_p(\hat{S}_{g,n},x)).$$ which depends on a $2p^{th}$ primitive root of unity $A_p$. We suppose, as in the proof of Theorem \ref{thm:infinite image}, that all the banded points on $(\hat{S}_{g,n},x)$ are colored by $2$. For compactness of notation, the space $V_p(\hat{S}_{g,n},x)$ will be denoted by $V_p(S)$. By Theorem \ref{thm:infinite image}, we may take $p$ such that $\rho_p$ has infinite image. In particular, $R := \ker \rho_p$ is an infinite index subgroup of $\pi_1(S)$.

Now we suppose that $p$ is a prime number such that $p \equiv 3 \, \pmod 4$ and we define $\zeta_p=A_p^2$ which is a $p^{th}$ root of unity. The integral TQFT as described in Subsection \ref{subsect:integral tqft} defines a representation 
\[\rho_{p} : \pi_1(S) \to  \text{PAut} (\mathcal{S}_p(S)),\] where $\mathcal{S}_p(S)$ is a free $\mathbb{Z}[\zeta_p]$--module of finite dimension $d_p$. If $k \geq 0$ is an integer, we can consider the representations \[\rho_{p,k} : \pi_1(S) \to \text{PAut} (\mathcal{S}_p(S) / h^{k+1} \mathcal{S}_p(S)),\]  where here $h = 1- \zeta_p$. Since the abelian groups $\mathcal{S}_p(S) / h^{k+1} \mathcal{S}_p(S)$ are finite, the groups $R_k := \ker \rho_{p,k}$ are finite index subgroups of $\pi_1(S)$.

Let $\mathcal{D}$ be the normal subgroup of $\pi_1(S)$ generated by
\[\{g^{n(g)} \mid g\in\pi_1(S) \, \text{simple} \, \text{and} \,  n(g) \, \text{the order of} \,  \rho_p(g) \}.\] Note that since simple elements of $\pi_1(S)$ have finite order image under $\rho_p$, the value of $n(g)$ is always finite, so that the definition of $\mathcal{D}$ makes sense.

Similarly for $k \geq 0$ let $\mathcal{D}_k$ be the normal subgroup generated by
\[\{ g^{n(g,k)} \mid g\in\pi_1(S) \, \text{simple} \, , \,  n(g,k) \, \text{ the order of} \,  g \, \, \text{in} \, \, \pi_1(S) / R_k \}.\]  Observe that if the subgroup $R_k<\pi_1(S)$ classifies a cover $S_k\to S$, then $\mathcal{D}_k$ is identified with the subgroup $\pi_1^s(S_k)<\pi_1(S)$.

We have the following filtration 

\[ \mathcal{D}  \subset R \subset\cdots\subset R_{k+1 }\subset R_{k} \subset\cdots\subset R_1 \subset R_0,\] and we have that \[\bigcap_k R_k=R\] (see Subsection \ref{subsect:integral tqft}), so that $ f \notin R$ if and only if  $ f \notin R_k$ for $k\gg 0$.

\begin{lem}\label{lem:right order}
For all $k\gg 0$, we have that $\mathcal{D}_k=\mathcal{D}$.
\end{lem}
\begin{proof}
If $g\in\pi_1(S)$ is simple then $\rho_p(g)$ has finite order $n(g)$. Since \[\bigcap_k R_k= R,\] we have that the order of the image of $g$ in $\pi_1(S)/R_k$ is exactly $n(g)$ for all $k\gg 0$. Finally, we have that $\rho_p$ is the restriction of a representation of the whole mapping class group $\Mod^1(S)$, under whose action there are only finitely many orbits of simple closed curves. In particular, for all $k\gg 0$ and all simple $g\in\pi_1(S)$, the order of the image of $g$ in $\pi_1(S)/R_k$ is exactly $n(g)$.
\end{proof}

With our choice of $p$ fixed, let us write $N_0$ for the smallest $k$ for which $\mathcal{D}_k=\mathcal{D}$, as in Lemma \ref{lem:right order}.
Let $\phi \in \pi_1(S)$ such that $\rho_p( \phi)$ has infinite order. There exists $m_0$ such that $\phi^{m_0} \in R_{N_0}$, since $R_{N_0}$ has finite index inside the group $\pi_1(S)$. For compactness of notation, we will write $\psi$ for $\phi^{m_0}$. Observe that $\psi \notin R$, since $\rho_p(\phi)$ has infinite order, and it follows that for $k\gg N_0$, we have that $\psi \notin R_k$. We set $N \geq N_0$ to be the integer such that $\psi \in R_N\setminus R_{N+1}$.  Notice that $R_N$ is a finite index subgroup of $\pi_1(S)$, whereas $\mathcal{D}_N = \mathcal{D}$ is an infinite index subgroup of $R_N$. In particular, $R_N$ can be naturally identified with the fundamental group of a finite regular cover $S_N\to S$, and $\mathcal{D}$ can be naturally identified with a subgroup of $\pi_1(S_N)$, i.e. the subgroup $\pi_1^s(S_N)$.

For each $N$, we may write \[q_N\colon R_N\to R_N/[R_N,R_N]\] for the abelianization map.

\begin{thm}\label{thm:not gen}
There is a proper inclusion \[q_N(\mathcal{D})\subsetneq R_N/[R_N,R_N].\] In fact, for every $\delta \in[R_N,R_N]$, we have that $\delta\cdot\psi\notin\mathcal{D}$.
\end{thm}

Theorem \ref{thm:not gen} implies Theorem \ref{thm:not gen homology} fairly quickly:

\begin{proof}[Proof of Theorem \ref{thm:not gen homology}]
Setting $S'$ in the statement of Theorem \ref{thm:not gen homology} to be the cover $S_N\to S$ classified by the subgroup $R_N$, we have that $\mathcal{D}$ can be identified with $\pi_1^s(S')$. The image of $\mathcal{D}$ in $R_N/[R_N,R_N]$ under $q_N$ is exactly $H_1^s(S',\Z)$. The conclusion of Theorem \ref{thm:not gen homology} now follows from Theorem \ref{thm:not gen}.
\end{proof}

\begin{proof}[Proof of Theorem \ref{thm:not gen}]
We will fix $\psi$ and $N$ as in the discussion before the theorem.
First, a standard and straightforward computation shows that $[R_N,R_N]\subset R_{2N+1}$.
Now, suppose there exists an element $\delta \in[R_N,R_N]$ such that $\delta\cdot\psi\in\mathcal{D}\subset R_{N+1}$. Then, we would obtain $\psi\in \delta^{-1}R_{N+1}$. Since $[R_N,R_N]\subset R_{2N+1}$ by the claim above, we have that $\delta\in R_{2N+1}$. But then we must have that $\psi\in R_{N+1}$ as well, which violates our choice of $N$, i.e. $\psi\in R_N\setminus R_{N+1}$.
\end{proof}

\subsection{Index estimates}
In this subsection, we estimate the index of $H_1^s(S_N,\Z)$ inside of $H_1(S_N,\Z)$ as a function of $p$ and of $N$. In particular, we will show that the index can be made arbitrarily large by varying both $p$ and $N$.

\begin{prop}\label{prop:index}
The index of $q_N(\mathcal{D})$ in $R_N/[R_N,R_N]$ is at least $p^e$, where here \[e =\left\lfloor \frac{N}{p-1} \right\rfloor+1.\]
\end{prop}

We need the following number--theoretic fact which can be found in \cite[Lemma 3.1]{MasbaumRoberts}:
\begin{lem} \label{lem1}
There exists an invertible element $z \in \mathbb{Z}[\zeta_p]$ such that $p = z \cdot h^{p-1}$.
\end{lem}

The following is an easy number--theoretic fact, whose proof we include for the convenience of the reader:

\begin{lem}
If $k$ and $p$ are relatively prime, then $k$ is invertible modulo $h^n$, for all $n\geq 1$.
\end{lem}
\begin{proof}
Since $p\equiv 0\pmod h$, we have that $p^n\equiv 0\pmod{h^n}$. Since $k$ and $p$ are relatively prime, we have that for each $n\geq 1$, there exist integers $a$ and $b$ such that \[a\cdot k+b\cdot p^n=1.\] Thus, $a\cdot k\equiv 1\pmod{h^n}$.
\end{proof}

\begin{lem}\label{lem:r2n+1}
Let $1\leq k\leq p^{e} -1$. With the notation of Theorem \ref{thm:not gen}, we have that $\psi^k\notin R_{2N+1}$.
\end{lem}
\begin{proof}
We compute the ``$h$--adic" expansion of $\rho_p(\psi)$ in a basis, as in the proof of Theorem \ref{thm:not gen}. Up to an invertible element of $\Z[\zeta_p]$, we have \[\rho_p(\psi)=I+h^{N+1}\Delta,\] where here $\Delta\nequiv 0\pmod h$, since $\psi\in R_N\setminus R_{N+1}$. If $1\leq k\leq p^e-1$, we obtain the expansion \[\rho_p(\psi)^k\equiv I+k\cdot h^{N+1}\Delta\pmod{h^{2N+2}}.\] Since $1\leq k\leq p^{e} -1$, we have that $k$ can be written $k = m \cdot p^l$, with $0 \leq l < e$ and with $m$ relatively prime to $p$.

Lemma \ref{lem1} implies that $p = z \cdot h^{p-1}$, so that 
 \[k \cdot h^{N+1}  \Delta =  m \cdot z^l \cdot h^{l(p-1)+N+1}\Delta.\]

Now if \[k \cdot h^{N+1}  \Delta \equiv 0 \pmod{h^{2N+2}},\] we see that \[ h^{l(p-1)+N+1}\Delta \equiv 0 \pmod{h^{2N+2}},\] since $m$ and $z$ are invertible modulo $h^{2N+2}$. Note however that $l < e$, so that $$l(p-1) < N+1.$$ 

The expression \[ h^{l(p-1)+N+1}\Delta \equiv 0 \pmod{h^{2N+2}}\] now implies that \[\Delta \equiv 0 \pmod{h^{N+1-l(p-1)}},\] which is impossible since $N+1-l(p-1) > 0$ and since $\Delta\nequiv 0\pmod h$.

It follows that \[k \cdot h^{N+1}  \Delta \nequiv 0 \pmod{h^{2N+2}},\] which in turn implies that $\psi^k\notin R_{2N+1}$.
\end{proof}

\begin{proof}[Proof of Proposition \ref{prop:index}]
In Theorem \ref{thm:not gen}, we established that for each \[ \delta \in [R_N,R_N]\subset R_{2N+1},\] we have that $\psi\cdot \delta\notin\mathcal{D}\subset R_{2N+1}$. From Lemma \ref{lem:r2n+1}, it follows that powers of $\psi$ represent at least $p^e$ distinct cosets of $R_{2N+1}$ in $R_N$, and hence of $[R_N,R_N]$ in $R_N$, whence the claim of the proposition.
\end{proof}

\section{Representations of surface groups, revisited}
In this final short section, we illustrate how a rational version of Theorem \ref{thm:not gen homology} implies Theorem \ref{thm:infinite image}, thus further underlining the interrelatedness of the two results. We will write $S$ for a surface as above.

\begin{prop}\label{prop:homology quotient}
Let $S'\to S$ be a finite regular cover of $S$, and suppose that $\rk(H_1^s(S',\Q))<\rk(H_1(S',\Q))$. Then there exists a linear representation \[\rho\colon \pi_1(S)\to \GL_d(\Z)\] such that:
\begin{enumerate}
\item
The image of $\rho$ is infinite.
\item
The image of every simple element of $\pi_1(S)$ has finite order.
\end{enumerate}
In fact, the image of $\rho$ can be virtually abelian (i.e. the image of $\rho$ has a finite index subgroup which is abelian).
\end{prop}

We remark again that our general version of Theorem \ref{thm:not gen homology} implies a proper inclusion between $H_1^s(S',\Z)$ and $H_1(S',\Z)$, which may no longer be proper when tensored with $\Q$. Therefore, we do not get that Theorem \ref{thm:infinite image} and Theorem \ref{thm:not gen homology} are logically equivalent. Indeed, in order to deduce Theorem \ref{thm:not gen homology}, we had to use specific properties of the $\SO(3)$--TQFT representations.

\begin{proof}[Proof of Proposition \ref{prop:homology quotient}]
Let $S'\to S$ be a finite regular cover as furnished by the hypotheses of the proposition, and let $G$ be the deck group of the cover. Write \[H_1(S',\Q)\cong A\oplus B,\] where $A\cong H_1^s(S',\Q)$ and $B\neq 0$. Note that the natural action of $G$ on $H_1(S',\Q)$ respects the summand $A$, since being simple is a conjugacy invariant in $\pi_1(S)$. Write $A^{\Z}$ and $B^{\Z}$ for the intersections of these summands with $H_1(S',\Z)$. Note that $H_1^s(S',\Z)\subset A^{\Z}$. Notice that for each integer $m\geq 1$, the subgroup $m\cdot A^{\Z}$ is characteristic in $A^{\Z}$, and is hence stable under the $G$--action on $H_1(S',\Z)$. Let $\gam$ be the group defined by the extension \[1\to H_1(S',\Z)\to \gam\to G\to 1,\] which is precisely the group \[\gam\cong \pi_1(S)/[\pi_1(S'),\pi_1(S')].\] Write $\gam_m=\gam/(m\cdot A^{\Z}$). This group is naturally a quotient of $\pi_1(S)$.

Note that for all $m$, the group $\gam_m$ is virtually a finitely generated abelian group, since it contains a quotient of $H_1(S',\bZ)$ with finite index. Note that every finitely generated abelian group is linear over $\Z$, as is easily checked. Moreover, if a group $K$ contains a finite index subgroup $H<K$ which is linear over $\Z$, then $K$ is also linear over $\Z$, as is seen by taking the induced representation. It follows that for all $m$, the group $\gam_m$ linear over $\Z$. Furthermore, there is a natural injective map $B^{\Z}\to \gam_m$, so that $\gam_m$ is infinite. Finally, if $g\in\pi_1(S)$ is simple, then for some $n=n(g)>0$, we have that $[g^n]\in H_1^s(S',\Z)\subset A^{\Z}$. It follows that $g$ has finite order in $\gam_m$. Thus, the group $\gam_m$ has the properties claimed by the proposition.
\end{proof}


\begin{thebibliography}{99}

\bibitem{AMU} J.E. Andersen, G. Masbaum, K. Ueno. Topological quantum field theory and the Nielsen-Thurston classification of $\mathrm{M}(0,4)$. \emph{Math. Proc. Cam. Phil. Soc. } 141(2006) 447--488.

\bibitem{Birman}
J. Birman. \emph{Braids, links, and mapping class groups}. 
Annals of Mathematics Studies, No. 82. \emph{Princeton University Press, Princeton, N.J.; University of Tokyo Press, Tokyo}, 1974.

\bibitem{bhmv}
C. Blanchet, N. Habegger, G. Masbaum, P. Vogel. Topological quantum field theories derived from the Kauffman bracket. \emph{Topology} 34(4) (1992), no. 4, 883--927.

\bibitem{Boggi1}
M. Boggi. The congruence subgroup property for the hyperelliptic modular group: the open surface case. \emph{Hiroshima Mathematical Journal}, 39(3):351--362, 2009.

\bibitem{Boggi2}
M. Boggi.  The congruence subgroup property for the hyperelliptic modular group: the closed surface case. Preprint, 2014.


\bibitem{delaHarpe}
P. de la Harpe. \emph{Topics in geometric group theory}.
Chicago Lectures in Mathematics. \emph{University of Chicago Press, Chicago, IL}, 2000.

\bibitem{FarbHensel}
B. Farb and S. Hensel. Moving homology classes in finite covers of graphs. Preprint, 2015.

\bibitem{farbmargalit}
B. Farb and D. Margalit. \emph{A primer on mapping class groups}. Princeton Mathematical Series, 49. \emph{Princeton University Press, Princeton, NJ}, 2012.

\bibitem{Funar}
L. Funar. On the TQFT representations of the mapping class groups. \emph{Pacific J. Math.} 188(1999), 251--274.

\bibitem{FunarKohno}
L. Funar and T. Kohno. Free subgroups within the images of quantum representations. \emph{Forum Math.} 26 (2014), no. 2, 337--355.

\bibitem{Gilmer}
P. Gilmer. Integrality for TQFTs. \emph{Duke Math. J.} 125 (2004), no. 2, 389--413.

\bibitem{GilmerMasbaum}
P. Gilmer and G. Masbaum. Integral lattices in TQFT. \emph{Ann. Sci. \'Ecole Norm. Sup. (4)} 40 (2007), no. 5, 815--844.

\bibitem{GilmerMasbaumTorus}
P. Gilmer and G. Masbaum. Integral TQFT for the one--holed torus. \emph{Pacific J. Math.} (252) no. 1 (2011), 93--112.


\bibitem{grunewaldlubotzkymalestein}
F. Grunewald, M. Larsen, A. Lubotzky, and J. Malestein. Arithmetic quotients of the mapping class group. \emph{Geom. Funct. Anal.}, to appear.

\bibitem{Irmer}
I. Irmer. Lifts of simple curves in finite regular coverings of closed surfaces. Preprint, 2015.


\bibitem{KentMarche MO}
R. Kent and J. March\'e, with comments by I. Agol. Math Overflow post, http://mathoverflow.net/questions/86894/homology-generated-by-lifts-of-simple-curves, 2012.


\bibitem{Looijenga} E. Looijenga. Some algebraic geometry related to the mapping class group. \emph{Oberwolfach Reports}, June 2015.


\bibitem{Masbaum infinite} G. Masbaum. An element of infinite order in TQFT-representations of mapping class groups. \emph{Low-dimensional topology (Funchal, 1998), volume 233 of Contemp. Math.}, pages 137--139. Amer. Math. Soc., Providence, RI, 1999.

\bibitem{MasbaumRoberts} G. Masbaum and J. D. Roberts. A simple proof of integrality of quantum invariants at prime roots of unity. \emph{Math. Proc. Cambridge Philos. Soc.} 121 (1997) no. 3, 443-454


\bibitem{MasbaumVogel} G. Masbaum and P. Vogel. $3$-valent graphs and the Kauffman bracket. \emph{Pacific J. Math.} 162(2) : 361--381, 1994


\bibitem{PutmanWieland}
A. Putman and B. Wieland. Abelian quotients of subgroups of the mapping class group and higher Prym representations. \emph{J. London Math. Soc.} (2) 88 (2013), no. 1, 79--96.

\bibitem{Raghunathan}
M.S. Raghunathan. \emph{Discrete subgroups of Lie groups}. Ergebnisse der Mathematik und ihrer Grenzgebiete, Band 68. \emph{Springer-Verlag, New York--Heidelberg}, 1972.


\bibitem{ReshetikhinTuraev}N. Reshetikhin and V. G. Turaev. Invariants of $3$-manifolds via link polynomials and quantum groups. \emph{Invent. Math.,} 103(3):547--597, 1991.


\bibitem{Turaev} V. G. Turaev. Quantum invariants of knots and $3$-manifolds. \emph{de Gruyter Studies in Mathematics}, vol. 18. \emph{Walter de Gruyter and Co., Berlin}, 2010.

\bibitem{Witten} E. Witten. Quantum field theory and the Jones polynomial.  \emph{Comm. Math. Phys.,} 121:351--399, 1989.
\end{thebibliography}
\end{document}